\documentclass [11pt,reqno]{amsart}

\usepackage{amsmath,amsfonts,amssymb,amscd,verbatim,delarray,verbatim,calrsfs}

\def\bbQ{{\mathbb Q}}
\def\bbZ{{\mathbb Z}}
\def\bbC{{\mathbb C}}
\def\bbR{{\mathbb R}}

\def\tra{\mathrm{tr}}

\def\End{\mathrm{End}}

\def\Hom{\mathrm{Hom}}

        \def\discr{\mathrm{discr}}
\def\GLGL{\mathrm{GL}}

\theoremstyle{plain}
\newtheorem{thm}{Theorem}[section]
\newtheorem{lem}[thm]{Lemma}
\newtheorem{cor}[thm]{Corollary}
\newtheorem{prop}[thm]{Proposition}

\theoremstyle{definition}
\newtheorem{defn}[thm]{Definition}
\newtheorem{ex}[thm]{Example}
\newtheorem{exs}[thm]{Examples}

           \newtheorem{rem}[thm]{Remark}

\begin{document}
\title[Odd and Even Elliptic Curves]
{Odd and Even Elliptic Curves with Complex Multiplication}
\author[Yuri\ G.\ Zarhin]{Yuri\ G.\ Zarhin}
\address{Department of Mathematics, Pennsylvania State University,
University Park, PA 16802, USA}

 \email{zarhin\char`\@math.psu.edu}
 \thanks{The  author  was partially supported by Simons Foundation Collaboration grant   \# 585711 
 and the Travel Support for Mathematicians Grant MPS-TSM-00007756 from the Simons Foundation. Part of this work was done
  in August-September 2025 during his stay at Max-Planck-Institut f\"ur Mathematik (Bonn, Germany), whose hospitality and support are gratefully acknowledged.}
 
 \begin{abstract}
 We call an order $O$ in a quadratic field  $K$ {\sl odd} (resp. {\sl even}) if its discriminant  is an odd (resp. 
 even) integer.
 We call an elliptic curve $E$ over  $\bbC$  with CM     {\sl odd} (resp. {\sl even}) if its endomorphism ring $\End(E)$ is an odd (resp. even)
 order in the imaginary quadratic field $\End(E)\otimes\bbQ$.

 Suppose that  $j(E)\in \bbR$ and   let us consider the set $\mathcal{J}(\bbR,E)$  of all $j(E^{\prime})$ where $E^{\prime}$ is any elliptic curve   that enjoys the following properties.
 
 \begin{itemize}
 \item
 $E^{\prime}$ is
 isogenous  to $E$;
 \item
    $j(E^{\prime})\in\bbR$;
    \item
 $E^{\prime}$ has the same parity as $E$.
 \end{itemize}
 
 We prove that the closure of $\mathcal{J}(\bbR,E)$ in $\bbR$  is the closed semi-infinite interval $(-\infty,1728]$  (resp. the whole $\bbR$) if $E$ is odd (resp. even). 
 
 This paper was inspired by  a question
 of Jean-Louis Colliot-Th\'el\`ene and Alena Pirutka about the distribution of $j$-invariants of certain  elliptic curves of CM type.
 \end{abstract}
 
 
\maketitle

As usual $\bbZ,\bbQ,\bbR,\bbC$ stand for the ring of integers and the fields of rational, real and complex numbers respectively.
If $n$ is a positive integer then we write $\bbZ[1/n]$ for the subring of $\bbQ$ generated by $\bbZ$ and $1/n$. More generally,
if $F$ is a field of characteristic zero and $A$ is a subring of $F$ (with $1$) then we write $A[1/n]$ for the subring 
 of $F$ generated by $A$ and $1/n$; clearly, $A[1/n]$ is a  $\bbZ[1/n]$-subalgebra of $F$. We write $\bbZ_{(2)}$  for the subring of $\bbQ$
 that consists of all fractions $\frac{a}{n}$ where $a \in \bbZ$ and $n$ an odd integer.

The aim of this note is to introduce and study the notion of a {\sl parity} for   elliptic curves over $\bbC$ with complex multiplication (CM) in the following way.
Let $E$ be an elliptic curve, whose endomorphism ring $\End(E)$ is an {\sl order} in an imaginary quadratic field.  We call $E$ {\sl odd} (resp. {\sl even}) if the discriminant of $\End(E)$ 
is an {\sl odd} (resp. {\sl even}) integer. We call its $j$-invariant  $j(E)\in \bbC$  {\sl odd} (resp. {\sl even}) if $E$ is  {\sl odd} (resp. {\sl even}).
It turns out that
both sets of even and odd  $j$-invariants are everywhere dense in $\bbC$.
Surprisingly,  this is not the case when we restrict ourselves to
 the  sets of even and odd  {\sl real } $j$-invariants (i.e., when the corresponding curves are defined over $\bbR$). Namely, while the set of   real {\sl even} $j$-invariants is dense in $\bbR$,
  the set of   real {\sl odd} $j$-invariants lies in an open semi-infinite interval $(-\infty, 1728)$ and everywhere dense there. Actually, we prove stronger versions of these results, dealing with isogeny classes of CM elliptic curves with the same ``parity''.
  
  The paper is organized as follows. In Section \ref{Orders} we discuss various  properties of quadratic orders with odd (resp. even) discriminants. Section \ref{Jinvariants} contains the statements of our main results about $j$-invariants of odd and even elliptic curves.
  In Section \ref{halfPlane} we discuss auxiliary results
about isogenies of complex elliptic curves. We prove our main results in Section \ref{mainproof},
using properties of ``odd'' and ``even'' orders and parity preservation under isogenies of odd degrees that will be proven
in Section \ref{OddEven} and \ref{OddEvenE} respectively. 
In Section \ref{revisit} we  ``classify'' (actually, list all the $j$-invariants of)   CM elliptic curves, whose $j$-invariant is real and
the endomorphism ring is an  order with given odd discriminant.
  
  {\bf Acknowledgements}.  I am grateful to Jean-Louis Colliot-Th\'el\`ene and Alena Pirutka for  an interesting and stimulating question
about the distribution of $j$-invariants of certain  elliptic curves of CM type and for their useful comments to a preliminary version of this paper,
which  may be viewed as an extension (or a variant) of my answer \cite[Sect. 8]{CTP}. I thank David Masser and Serge Vladuts for their interest in this paper and stimulating
discussions. My special thanks go to the referee, whose comments helped to improve the exposition.
  
  \section{Orders in quadratic fields}
  \label{Orders}
  Let us start with a discussion of {\sl odd} and {\sl even} orders in arbitrary quadratic fields.

 Let $K$ be a quadratic field, $O_K$ its ring of integers, and 
 $$\mathrm{tr}=\tra_K: K \to \bbQ$$
   be the corresponding trace map attached to the field extension $K/\bbQ$.
 We have 
 $$2\bbZ=\tra(\bbZ)\subset \mathrm{tr}(O_K)\subset \bbZ.$$  
 Let $O$ be an order in $K$ and  let  $\mathfrak{f}=\mathfrak{f}_O$ denote the  {\sl conductor} of $O$, which is
a positive integer.  We have
$$O=\bbZ+\mathfrak{f}\cdot O_K\subset O_K;$$ 
$\mathfrak{f}$ coincides with the index $[O_K:O]$ of the additive subgroup $O$ in $O_K$
\cite[Ch. 2, Sect. 7, Th.1]{BS}.
Since  $\bbZ\subset O$ and $\mathrm{tr}(\bbZ)=2\bbZ$,
\begin{equation}
\label{trO2}
 \mathrm{tr}(O)=\bbZ \quad \text{or} \ 2\bbZ.
 \end{equation}

Recall that the discriminant $\mathrm{discr}(O)\in \bbZ$ of $O$ is the discriminant of the symmetric bilinear form
$$O \times O \to \bbZ, \ a,b \mapsto \mathrm{tr}(ab).$$
By definition, the discriminant $\mathrm{discr}(K)$  of the field $K$ is $\mathrm{discr}(O_K)$.  It is well known \cite[Ch. 2, Sect. 7, Th.1]{BS} that
\begin{equation}
\label{disF}
\mathrm{discr}(O)=\mathfrak{f}_O^2 \cdot \mathrm{discr}(O_K)=\mathfrak{f}_O^2\cdot \mathrm{discr}(K).
\end{equation}

\begin{defn}
\label{DefOd}
An order $O$ is called {\sl odd} if it enjoys the following equivalent properties.
\begin{itemize}
\item[(i)]
$\mathrm{tr}(O)=\bbZ$.
\item[(ii)]
$\mathrm{discr}(O)$ is an odd integer.
\item[(iii)]
Both $\mathrm{discr}(K)$ and the conductor $\mathfrak{f}_O$  are odd integers.
\item[(iv)]
There exists an integer $D \equiv 1 (\bmod 4)$ that is not a square and such that
$$\frac{1+\sqrt{D}}{2} \in O.$$
\item[(v)]
There exists an integer $D \equiv 1 (\bmod 4)$ that is not a square such that
$$K=\bbQ(\sqrt{D}), \quad
 O=\bbZ+\bbZ \frac{1+\sqrt{D}}{2}=\bbZ\left[\frac{1+\sqrt{D}}{2}\right].$$
\end{itemize}
\end{defn}

\begin{defn}
\label{DefEv}
An order $O$ in a quadratic field $K$ is called {\sl even} if it enjoys the following equivalent properties.
\begin{itemize}
\item[(i)]
$\mathrm{tr}(O)=2\bbZ$.
\item[(ii)]
$\mathrm{discr}(O)$ is an even integer.
\item[(iii)]
Either $\mathrm{discr}(K)$ or the conductor $\mathfrak{f}_O$  is an even integer.
\end{itemize}
\end{defn}

Clearly, every order $O$ is either odd or even. We prove the equivalences presented in Definition \ref{DefOd} and 
the equivalences presented in 
Definition \ref{DefEv} in Section \ref{OddEven}.

\begin{defn}
Let $E$ be an elliptic curve with complex multiplication (CM) over the field $\bbC$ of complex numbers.
Then its endomorphism ring $\End(E)$ is an order in the imaginary quadratic field $$K:=\End(E)\otimes\bbQ=
:\End^0(E).$$
We say that $E$ is {\sl odd} (resp. {\sl even}) if the  order $\End(E)$  is {\sl odd} (resp. {\sl even}).
\end{defn}

Clearly, every elliptic curve over $\bbC$ with CM is either odd or even. The following assertions will be proven in Section \ref{OddEvenE}.

\begin{prop}
\label{isogPar}
Let $\phi: E_1 \to E_2$ be an isogeny of complex elliptic curves with CM.
Suppose that $n=\deg(\phi)$ is an {\sl odd} integer.
Then $E_1$ is {\sl odd} (resp. {\sl even}) if and only if $E_2$ is {\sl odd} (resp. {\sl even}).
In other words, $E_1$ and $E_2$ have the same parity.
\end{prop}

\begin{prop}
\label{isogR}
Suppose that $E_1$ and $E_2$ are isogenous  CM elliptic curves  over $\bbC$.  Suppose that there are elliptic curves 
$E_{1,\bbR}$ and $E_{2,\bbR}$ over the field $\bbR$ of real numbers such that there are isomorphisms of complex elliptic curves
$$E_1 \cong E_{1,\bbR} \times_{\bbR}\bbC, \quad E_2=E_{2,\bbR} \times_{\bbR}\bbC.$$
Then $E_{1,\bbR}$ and $E_{2,\bbR}$  are isogenous over $\bbR$.
\end{prop}

\begin{rem}
\label{defR}
Let $E$ be an elliptic curve over $\bbC$. It is well known  \cite[Ch. 3, Sect. 2, Prop. 3.7]{Knapp} that $j(E) \in \bbR$ if and only if $E$ may be defined over $\bbR$, i.e.,
there is an elliptic curve $E_{\bbR}$ over $\bbR$ such that $E \cong E_{\bbR} \times_{\bbR}\bbC$. (See also \cite[Appendix A, Prop. 1.2(b)]{Sil1} and \cite[Ch. 4, Sect. 4.1]{Shi}).)
\end{rem}

\begin{exs}
\label{ExOE}
\begin{itemize}
\item[(e)]
If $E$ is an elliptic curve $y^2=x^3-x$ then it is well known 
(e.g., see \cite[Ch. III, Sect. 4, Example 4.4]{Sil1}) that the order $\End(E)=\bbZ[\sqrt{-1}]$ has even discriminant $-4$ and therefore $E$ is {\sl even}.
\item[(o)]
If $E$ is an elliptic curve $y^2=x^3-1$  then it is  easy to check that $E$ admits an automorphism of order 3.
This implies (see \cite[Ch. I, Sect. 5, p. 20]{LangEF}) the well known assertion that  
the order $\End(E)=\bbZ\left[\frac{-1+\sqrt{-3}}{2}\right]$ has odd discriminant $-3$ and therefore $E$ is 
{\sl odd}.
\end{itemize}
\end{exs}

\section{$j$-invariants of CM elliptic curves}
\label{Jinvariants}
In what follows, $E$ is an elliptic curve  over  $\bbC$. 
The following assertion will be proven in Section \ref{halfPlane}.

\begin{prop}
\label{orbit}
 Let us consider
the subset $J_{E, \mathrm{is}}$ of all $j(E^{\prime})$ where $E^{\prime}$ runs through the set of  elliptic 
curves over $\bbC$ such that there exists an isogeny $E \to E^{\prime}$ of odd degree.

Then  $J_{E, \mathrm{is}}$ is dense in $\bbC$.
\end{prop}

\begin{rem}
If $E$ has CM then
all $E^{\prime}$ that appear in Proposition \ref{orbit} are elliptic curves with CM
and the corresponding imaginary quadratic fields $\End^{0}(E^{\prime})$ are isomorphic to $\End^0(E)$.
\end{rem}

Notice that in light of Examples \ref{ExOE}, there exist an odd CM curve and an even CM curve. Now,
combining Proposition \ref{orbit} with Proposition \ref{isogPar}, we obtain the following statement.

\begin{cor}
\label{orbitPar}
Let $J^{\mathrm{odd}}(\bbC)\subset \bbC$ be the set of all $j(E^{\mathrm{odd}})$ where $E^{\mathrm{odd}}$ runs through the set
of all odd elliptic curves over $\bbC$ with CM.

Let $J^{\mathrm{ev}}(\bbC)\subset \bbC$ be the set of all $j(E^{\mathrm{ev}})$ where $E^{\mathrm{ev}}$ runs through the set
of all  even elliptic curves  over $\bbC$ with CM.

Then each of two subsets $J^{\mathrm{odd}}(\bbC)$ and $J^{\mathrm{ev}}(\bbC)$ is dense in $\bbC$.
\end{cor}

Our main result is the following assertion.

\begin{thm}
\label{orbitR}
 Suppose that  $E$ is a CM curve with $j(E)\in \bbR$. Let us   consider the set $\mathcal{J}(\bbR,E)$  of all $j(E^{\prime})$ where $E^{\prime}$ is any elliptic curve  over $\bbC$ 
 that enjoys the following properties.
 \begin{itemize}
 \item
 $E^{\prime}$ is isogenous 
 to $E$;
 \item
  $j(E^{\prime})\in \bbR$;
  \item
   $E^{\prime}$ has the same parity as $E$.
 \end{itemize}
 Then the closure of $\mathcal{J}(\bbR,E)$ in $\bbR$  is the closed semi-infinite interval $(-\infty,1728]$  (resp. the whole $\bbR$) if $E$ is odd (resp. even). 
\end{thm}

We prove Theorem \ref{orbitR} in Section \ref{mainproof}. Our proof is based on auxiliary results about isogenies of complex elliptic curves with CM that will be
discussed in Section \ref{halfPlane}.

\begin{rem}
It follows from Remark \ref{defR} and Proposition \ref{isogR} that Theorem \ref{orbitR}.
is equivalent to the following statement.
\end{rem}

\begin{thm}
\label{orbitRR}
Let $E_{\bbR}$ be an elliptic curve over $\bbR$ such that its complexification $E=E_{\bbR}\times_{\bbR}\bbC$ is of CM type. 
 Let us   consider the set $\mathcal{J}_0(\bbR,E_{\bbR})$  of all $j(E_{\mathbb{R}}^{\prime})$ where $E^{\prime}_{\bbR}$ is any elliptic curve  over $\bbR$ that 
 enjoys the following properties.
 \begin{itemize}
 \item
 $E^{\prime}_{\bbR}$ is $\bbR$-isogenous  to $E_{\bbR}$;
 \item
 The complexification $E^{\prime}:=E^{\prime}_{\bbR}\times_{\bbR}\bbC$ of  $E^{\prime}_{\bbR}$
  has the same parity as $E$.
 \end{itemize}
 Then the closure of $\mathcal{J}_0(\bbR,E_{\bbR})$ in $\bbR$  is the closed semi-infinite interval $(-\infty,1728]$  (resp. the whole $\bbR$) if $E$ is odd (resp. even). 
\end{thm}

\section{The upper half-plane and elliptic curves}
\label{halfPlane}
\begin{defn}
If $\omega_1$ and $\omega_2$ are nonzero complex numbers such that $\omega_1/\omega_2 \not\in \bbR$ then we write
$[\omega_1,\omega_2]$ for the discrete lattice $\bbZ\omega_1+\bbZ\omega_2$ of rank $2$ in $\bbC$.
\end{defn}
Let
$$\mathfrak{H}:=\{x+\mathbf{i}y \mid x,y \in \mathbb{R}; \ y>0\}\subset \bbC$$
be the {\sl upper half-plane}. If $\tau \in \mathfrak{H}$ then we write $\mathcal{E}_{\tau}$ for the complex elliptic curve
such that the complex tori $\mathcal{E}_{\tau}(\bbC)$ and $\bbC/\Lambda_{\tau}$ are isomorphic. Here
$$\Lambda_{\tau}:=[\tau,1]=\bbZ\tau+\bbZ.$$
See \cite[Ch. V, p. 408--411]{Sil} for a Weierstrass equation of $\mathcal{E}_{\tau}$. (In the notation of \cite[Ch. V]{Sil}, our $\mathcal{E}_{\tau}$ is $E_q$ with $q=\exp(2\pi \mathbf{i}\tau)$.)
\begin{rem}
\label{isogE1E2}
If $\tau_1,\tau_2 \in \mathfrak{H}$ then the group $\Hom(\mathcal{E}_{\tau_1},\mathcal{E}_{\tau_2})$ of homomorphisms from 
$\mathcal{E}_{\tau_1},\mathcal{E}_{\tau_2}$ can be be canonically identified with
\begin{equation}
\label{tau1tau2}
\{u \in \bbC \mid u(\Lambda_{\tau_1})\subset \Lambda_{\tau_2}\}\subset \bbC.
\end{equation}
Each  $u$ that satisfies \eqref{tau1tau2} corresponds to the homomorphism
$$\phi_{u}: \mathcal{E}_{\tau_1} \to \mathcal{E}_{\tau_2}$$ of elliptic curves such that the corresponding action of $\phi_{u}$
 on the complex points is as follows.
$$\mathcal{E}_{\tau_1}(\bbC)=\bbC/\Lambda_{\tau_1} \to \bbC/\Lambda_{\tau_2} =\mathcal{E}_{\tau_2}(\bbC), \quad z+\Lambda_{\tau_1} \mapsto u z+\Lambda_{\tau_2}.$$
If $u \ne 0$ then $\phi_{u}$ is an isogeny, 
whose degree coincides with the index $[\Lambda_{\tau_2}:u(\Lambda_{\tau_1})]$  of the subgroup 
$u(\Lambda_{\tau_1})$ in  $\Lambda_{\tau_2}$ \cite[p. 9]{Gross}. In particular, if 
$$\Lambda_{\tau_2}=u(\Lambda_{\tau_1})$$
(i.e., the index is $1$), $\phi_u$ is an isomorphism of elliptic curves $\mathcal{E}_{\tau_1}$ and $\mathcal{E}_{\tau_2}$.

On the other hand, if we take $\tau_1=\tau_2=\tau$ and put 
\begin{equation}
\label{LambdaO}
O_{\tau}:=\{u \in \bbC\mid u(\Lambda_{\tau})\subset \Lambda_{\tau}\},
\end{equation}
then the map
\begin{equation}
\label{Oend}
O_{\tau} \to \End(\mathcal{E}_{\tau}), \ u \mapsto \phi_u
\end{equation}
is a {\sl ring isomorphism}.
\end{rem}
\begin{rem}
\label{LO}
We have
$$\bbZ\subset O_{\tau}, \quad \Lambda_{\tau}\supset O_{\tau} \cdot 1=O_{\tau}\subset \bbC,$$
i.e., 
\begin{equation}
\label{LO2}
\bbZ\subset O_{\tau}\subset \Lambda_{\tau}\subset \bbC.
\end{equation}
\end{rem}
It follows easily that if $\Lambda_{\tau}$ is a {\sl subring} of $\bbC$ then  $O_{\tau}= \Lambda_{\tau}$.
\begin{rem}
\label{OrderTau}
\begin{itemize}
\item[(i)]
It is well known \cite[Ch. 4, Sect. 4.4, Prop. 4.5]{Shi} that $\mathcal{E}_{\tau}$ has CM if and only if $\tau$ is an (imaginary) quadratic irrationality,
i.e., $\bbQ(\tau)=\bbQ+\bbQ\tau$ is an imaginary quadratic field. If this is the case then 
$O_{\tau}$ is an order in the imaginary quadratic field  $\bbQ(\tau)$; in particular,
\begin{equation}
\label{CMtau}
O_{\tau}=\{u \in \bbQ(\tau)\mid u(\Lambda_{\tau}) \subset \Lambda_{\tau}\}.
\end{equation}

 In addition, extending isomorphism \eqref{Oend} by $\bbQ$-linearity,
we get a canonical isomorphism of $\bbQ$-algebras
$$\bbQ(\tau) =O_{\tau} \otimes \bbQ=\bbQ(\tau) \to \End(\mathcal{E}_{\tau})\otimes\bbQ=\End^0(\mathcal{E}_{\tau}), \quad u\otimes r \mapsto \phi_u \otimes r$$
between $\bbQ(\tau)$ and the endomorphism algebra $\End^0(\mathcal{E}_{\tau})$ of $\mathcal{E}_{\tau}$. 
\item[(ii)]
Suppose that $\mathcal{E}_{\tau}$ has CM, i.e.,  $\bbQ(\tau)$ is an imaginary quadratic field. In light of \eqref{LO2},
\begin{equation}
\label{LO3}
\bbZ\subset O_{\tau}\subset \Lambda_{\tau}=\bbZ\tau+\bbZ\subset \bbQ(\tau).
\end{equation}
It follows from Remark \ref{LO} that if $ \Lambda_{\tau}$ is an {\sl order} in $\bbQ(\tau)$ then
$$O_{\tau}=\Lambda_{\tau}.$$
\item[(ii)]
Let $\mathcal{E}_{\tau_1}$ and $\mathcal{E}_{\tau_2}$ be two elliptic curves with complex multiplication. Then $\mathcal{E}_{\tau_1}$ and $\mathcal{E}_{\tau_2}$
are isogenous if and only if the corresponding imaginary quadratic fields coincide, i.e.,
$$\bbQ(\tau_1)=\bbQ(\tau_2)$$
(see \cite[Ch. 4, Sect. 4.4, Prop. 4.9]{Shi}).
\end{itemize}
\end{rem}

The group $\GLGL_2(\mathbb{R})^{+}$ of two-by-two real matrices with positive determinant acts transitively on $\mathfrak{H}$
by fractional-linear transformations. Namely, the continuous map
$$M, \tau \mapsto M(\tau)=\frac{a\tau+b}{c\tau+d} \quad \forall   \tau \in \mathfrak{H}, \  M=
\begin{pmatrix}
a & b\\
c & d
\end{pmatrix}
\in \GLGL_2(\mathbb{R})^{+}.$$
defines the transitive action of $\GLGL_2(\mathbb{R})^{+}$ on  $\mathfrak{H}$.
(Actually, even the action of the subgroup $\mathrm{SL}(2,\bbR)$ on  $\mathfrak{H}$ is transitive.)  We will mostly deal with the action of  the subgroup
$$G:=\GLGL(2,\bbZ_{(2)})^{+}=\{M \in \mathrm{GL}(2,\bbZ_{(2)})\mid \det(M)>0\}\subset \mathrm{GL}_2(\mathbb{R})^{+}.$$

\begin{lem}
\label{oddG}

Suppose that  a matrix
$$M=\begin{pmatrix}
a & b\\
c & d
\end{pmatrix}
\in \GLGL_2(\bbZ_{(2)})$$ has positive determinant 
$$\det(M)=ad-bc\in \bbZ_{(2)}^{*},$$
i.e, $M \in G$.
Then there exists an isogeny $\mathcal{E}_{M(\tau)} \to \mathcal{E}_{\tau}$ of odd degree.
\end{lem}

\begin{proof}
By definition of $\bbZ_{(2)}$, there are {\sl integers} $\tilde{a},\tilde{b},\tilde{c}, \tilde{d}$,
and a positive {\sl odd} integer $n$ such that
$$a=\tilde{a}/n, \ b=\tilde{b}/n, \ c=\tilde{c}/n, \ d=\tilde{d}/n.$$
The conditions on $\det(M)$ mean that 
$ad-bc=D_1/D_2$ where both $D_1$ and $D_2$ are {\sl positive odd} integers. This implies that
$$\bbZ \ni \tilde{a}\tilde{d}-\tilde{b}\tilde{c}=(na)(nb)-(nc)(nd)=
n^2(ad-bc)=n^2 \frac{D_1}{D_2}=\frac{n^2 D_1}{D_2}.$$
It follows that $\tilde{a}\tilde{d}-\tilde{b}\tilde{c}$ is a positive integer that divides the odd integer $n^2 D_1$
(recall that both $n$ and $D_1$ are odd).  It follows that  $\tilde{a}\tilde{d}-\tilde{b}\tilde{c}$ is a {\sl positive odd} integer.
Then
$$M(\tau)=\frac{a\tau+b}{c\tau+d}=\frac{\tilde{a}\tau+\tilde{b}}{\tilde{c}\tau+\tilde{d}}=\tilde{M}(\tau)$$
where the matrix
$$\tilde{M}=\begin{pmatrix}
\tilde{a} & \tilde{b}\\
\tilde{c} & \tilde{d}
\end{pmatrix}
.$$
(Since 
$$\tilde{a}\tilde{d}-\tilde{b}\tilde{c}=\det(\tilde{M})=\frac{1}{n^2}\det(M)>0,$$  we get 
$$\tilde{M} \in \GLGL(2,\bbZ_{(2)})^{+}\subset \GLGL_2(\bbR)^{+}.)$$
We have
$$\Lambda_{M(\tau)}=\Lambda_{\tilde{M}(\tau)}=\left[\frac{\tilde{a}\tau+\tilde{b}}{\tilde{c}\tau+\tilde{d}},\ 1\right].$$
This means that
$$(\tilde{c}\tau+\tilde{d})\Lambda_{M(\tau)}=[\tilde{a}\tau+\tilde{b}, \ \tilde{c}\tau+\tilde{d}]$$
is a subgroup in $[\tau,1]=\Lambda_{\tau}$ of odd index $\tilde{a}\tilde{d}-\tilde{b}\tilde{c}$.
This implies that if we put 
$$u:=\tilde{c}\tau+\tilde{d}\in \bbC, \quad  k:=\tilde{a}\tilde{d}-\tilde{b}\tilde{c}$$ then
$u\left(\Lambda_{M(\tau)}\right)$ is a subgroup of {\sl odd} index $k$ in $\Lambda_{\tau}$. This gives us the isogeny
$\phi_u: \mathcal{E}_{M(\tau)} \to \mathcal{E}_{\tau}$ of odd degree $k$. This ends the proof.
\end{proof}

\begin{proof}[Proof of Proposition \ref{orbit}]
Let $B \in\GLGL(2,\bbR)^{+}$. The  {\sl weak approximation} for the field $\bbQ$ \cite[Th. 1]{Artin} with respect to places $\{\infty,\ 2\}$
implies the existence of a sequence $\{B_n\}$ of $2 \times 2$ matrices with rational entries such that
$\{B_n\}$ converges to $B$ in the real topology and to the identity matrix 
$\begin{pmatrix}
1 & 0\\
0 & 1
\end{pmatrix}
$
in the $2$-adic topology. Removing first finite terms of the sequence, we may and will assume that for all $n$
$$B_n \in \GLGL(2,\bbZ_{(2)})\subset \GLGL(2,\bbQ), \quad \det(B_n)>0;$$
the latter inequality means that all  $B_n \in \GLGL(2,\mathbb{R})^{+}$ and therefore
$$B_n \in  \GLGL(2,\bbZ_{(2)})^{+}=G.$$
It follows that $B(\tau)\in \mathfrak{H}$ is the limit of the sequence 
$$B_n(\tau) \in \mathfrak{H}$$
in the complex topology. Now the transitivity of  $\GLGL(2,\mathbb{R})^{+}$ implies that 
every orbit
$G\tau=\{M(\tau) \mid M \in G\}$ is dense in $\mathcal{H}$
in the complex topology.

Recall that the classical holomorphic (hence, continuous) modular function
$$j:\mathfrak{H} \to \bbC$$
takes on every complex value. It follows that for every $\tau$ the set
$\{j(M(\tau))\mid M \in G\}$ is dense in $\bbC$.

Let $E$ be an elliptic curve with complex multiplication. Then there exists $\tau\in H$ such that
the elliptic curves $E$ and $\mathcal{E}_{\tau}$ are isomorphic and therefore
$$j(E)=j(\mathcal{E}_{\tau})=j(\tau).$$
In light of Lemma \ref{oddG}, if $M \in G$ then there is an 
isogeny $\mathcal{E}_{M(\tau)} \to E$ of odd degree. Since
$$j(\mathcal{E}_{M(\tau)})=j(M(\tau)),$$
the set $J_{E, \mathrm{is}}$ of complex numbers contains 
$\{j(M(\tau))\mid M \in G\}$, which  is dense in $\bbC$. Hence, $J_{E, \mathrm{is}}$ is also dense in $\bbC$,
which ends the proof.
\end{proof}

\begin{ex}
\label{crucialE}
Let $\tau \in \mathfrak{H}$.  Let $n, m$ be  positive {\sl odd} integers.
Then the matrix
$$M=\begin{pmatrix}
m & 0\\
0 & n
\end{pmatrix}
$$
 satisfies the conditions of Lemma \ref{oddG}. We have
$$M(\tau)=\frac{m}{n} \tau.$$
It follows from Lemma \ref{oddG} that there exists an isogeny
$\mathcal{E}_{\frac{m}{n} \tau} \to \mathcal{E}_{\tau}$
 of {\sl odd} degree. 
 This implies that  if $\tau$ is an imaginary quadratic irrationality then, by Proposition 1.4 (to be proven in Section \ref{OddEvenE}),
  the CM elliptic curves $\mathcal{E}_{\frac{m}{n}\tau}$ and $\mathcal{E}_{\tau}$ have the same parity.
\end{ex}

\begin{ex}
\label{crucial}
Let $\tau=\frac{1+\mathbf{i}y}{2} \in \mathfrak{H}$ where $y \in \bbR_{+}$. Let $n, m$ be  positive odd integers. Let us put
$$\tau_{m,n}:=\frac{1+\mathbf{i}\frac{m}{n}y}{2} \in \mathfrak{H}.$$
Notice that $n-m$ is even and therefore  $\frac{n-m}{2}\in \bbZ$.
Then the matrix
$$M=\begin{pmatrix}
m & \frac{n-m}{2}\\
0 & n
\end{pmatrix}
$$
 satisfies the conditions of Lemma \ref{oddG}. We have
$$M(\tau)=
m\cdot \frac{1+\mathbf{i}y}{2n}+\frac{n-m}{2n}=
\frac{n}{2n}+m\cdot \frac{\mathbf{i}y}{2n}=\frac{1+\mathbf{i}\frac{m}{n}y}{2}=
\tau_{m,n}.$$
It follows from Lemma \ref{oddG} that there exists an isogeny
$\mathcal{E}_{\tau_{m,n}} \to \mathcal{E}_{\tau}$
 of {\sl odd} degree. 
  By Proposition 1.4 (to be proven in Section \ref{OddEvenE}),  if $\tau$ is an imaginary quadratic irrationality then the CM elliptic curves $\mathcal{E}_{\tau_{m,n}}$ and $\mathcal{E}_{\tau}$ have the same parity.
\end{ex}

In what follows, if $D$ is a negative real number then we write $\sqrt{D}$ for $\sqrt{|D|}\cdot \mathbf{i}\in \mathfrak{H}$.

\begin{ex}
\label{EoddD}

If $\tau$ is a point in  $\mathfrak{H}$ such that $\bbQ(\tau)$ is an imaginary quadratic field and the lattice $\Lambda_{\tau}$ is an {\sl order} in $\bbQ(\tau)$
then $O_{\tau}=\Lambda_{\tau}$.  E.g., let  $D$ be a negative integer that is congruent to $1$ modulo $4$ and
$$\tau=\frac{1+\sqrt{D}}{2} \in \mathfrak{H}.$$
Then
$\Lambda_{\tau}=\bbZ+\bbZ \frac{1+\sqrt{D}}{2}$ is an order in the imaginary quadratic field $\bbQ(\tau)=\bbQ(\sqrt{D})$.
By Remark \ref{OrderTau}(ii),
$$O_{\tau}=\Lambda_{\tau}=\bbZ+\bbZ \frac{1+\sqrt{D}}{2}.$$
It follows from Definition \ref{DefOd} that
$O_{\tau}$
 is an {\sl odd}  order in the quadratic field $\bbQ(\sqrt{D})=\bbQ(\tau) \cong \End^0(\mathcal{E}_{\tau})$. This means that the elliptic curve $\mathcal{E}_{\tau}$ is odd.
 
 Now let $n, m$ be any  positive odd integers. Let us consider the complex numbers
 \begin{equation}
 \label{tauMND}
\tau_{m,n}(D)=\frac{1+\mathbf{i}\frac{m}{n}\sqrt{|D|}}{2} =\frac{1+\frac{m}{n}\sqrt{D}}{2}\in \mathfrak{H}.
\end{equation}
It follows from Example \ref{crucial} (applied to $y=\sqrt{|D|}$) that all the elliptic curves $\mathcal{E}_{\tau_{m,n}(D)}$ are also odd.
\end{ex}

\begin{lem}[Key Lemma]
\label{realDIV}
 Let $\tau \in \mathfrak{H}$ such that  $\bbQ(\tau)$ is an imaginary
quadratic field, i.e., $\mathcal{E}_{\tau}$  is an  elliptic curve with CM.

\begin{itemize}
\item[(i)]
If $\mathrm{Re}(\tau)=0$ then $\mathcal{E}_{\tau}$  is even.
\item[(ii)]
 Suppose that $\mathrm{Re}(\tau)=1/2$.  Let $D$ be a negative integer such that 
 $\bbQ(\tau)=\bbQ(\sqrt{D})$ and $D \equiv 1 (\bmod 4)$. 
 Then
 \begin{equation}
 \label{DtauO}
 \frac{1+\sqrt{D}}{2} \in O_{\tau}
 \end{equation} 
 if and only if there is a positive odd integer $\beta$ such that
$\beta$ divides $D$ and
$$\tau=1/2+\frac{1}{2\beta}\sqrt{D}=\frac{1+\sqrt{D}/\beta}{2}=\tau_{1,\beta}(D).$$
If this is the case then $\mathcal{E}_{\tau}$ is odd.
\item[(iii)]
Let $D$ be a negative integer such that 
$D \equiv 1 (\bmod 4)$. Let $\beta$ be a positive odd integer and
$\mathbf{d}=\mathbf{d}_{\beta,D}\ge 1$ be the greatest common divisor of $D$ and $\beta^2$.
 Let us put
$$\tau:=\tau_{1,\beta}(D)=\frac{1+\frac{1}{\beta}\sqrt{D}}{2}
=\frac{\beta+\sqrt{D}}{2\beta}\in \mathfrak{H}.$$

Then $\mathcal{E}_{\tau}$ is odd,
$\bbQ(\tau)=\bbQ(\sqrt{D})$ is an imaginary quadratic field and the order
$O_{\tau}$   coincides with
$$\bbZ+\bbZ \left(\frac{\beta}{\mathbf{d}}\right) \left(\frac{\beta+\sqrt{D}}{2\beta}\right)$$
and has discriminant
$$\left(\frac{\beta}{\mathbf{d}}\right)^2 D=\frac{\beta^2}{\mathbf{d}} \cdot  \frac{D}{\mathbf{d}}.$$
\end{itemize}
\end{lem}

\begin{proof}

{\bf  (i)}. Assume that $\mathrm{Re}(\tau)=0$.  Suppose that $\mathcal{E}_{\tau}$ is odd, i.e., the order $O_{\tau}$ is odd. It follows 
from Definition \ref{DefOd}
that there is a {\sl negative} integer $$D \equiv 1 (\bmod 4)$$ 
such that
$$\frac{1+\sqrt{D}}{2} \in O_{\tau}\subset  K=\bbQ(\sqrt{D})=\bbQ+\bbQ \sqrt{D}.$$
 Since $\tau \in \bbQ(\sqrt{D})=\bbQ+\bbQ\cdot \sqrt{D}$ and $\mathrm{Re}(\tau)=0$, there is a positive rational number $r$ such that
 $$\tau=r \sqrt{D}.$$
 This implies that 
 $$\Lambda_{\tau}=\bbZ+\bbZ \cdot r\sqrt{D}.$$
 It follows from the definition of $O_{\tau}$ (see  \eqref{LambdaO}) that
 $$\frac{1+\sqrt{D}}{2} \Big(\Lambda_{\tau}\Big)\subset \Lambda_{\tau}.$$
 This implies that there are integers $\alpha,\beta \in \bbZ$ such that
 $$\frac{1+\sqrt{D}}{2}\cdot 1=\alpha+\beta  r\sqrt{D}.$$
 Taking the real parts of both sides, we get $\alpha=1/2$, which contradicts the integrality of $\alpha$.
 The obtained contradiction proves that $\mathcal{E}_{\tau}$ is even, which proves (i).
 
 {\bf (iii)}. It follows from Example \ref{EoddD} that $\mathcal{E}_{\tau}$ is odd.
 
 Since $\beta$ is odd, $\beta^2 \equiv 1 \ (\bmod 4)$. This implies that
 $4 | (\beta^2-D)$, because $D \equiv 1 \ (\bmod 4)$. On the other hand, $\mathbf{d} | (\beta^2-D)$,
 because both $\beta^2$ and $D$ are divisible by $\mathbf{d}$.
 Since $\mathbf{d}$ divides $D$, it is also odd. Since odd $\mathbf{d}$ and $4$ are relatively prime, their product $4\mathbf{d}$ also divides $\beta^2-D$, i.e.,
 $$c: =\frac{\beta^2-D}{4\mathbf{d}} \in \bbZ.$$
 Let us put
 $$a:=\frac{\beta^2}{\mathbf{d}}\in \bbZ, \quad b:=-\frac{\beta^2}{\mathbf{d}}=-a\in \bbZ.$$
 I claim that the greatest common divisor  (GCD) of three integers $a,b,c$ is $1$.
 Indeed, by the definition of $\mathbf{d}$, the integers $\beta^2/\mathbf{d}$ and $D/\mathbf{d}$ are relatively prime.
 This implies that $a=\beta^2/\mathbf{d}$ and $4c=(\beta^2-D)/\mathbf{d}$ are relatively prime. Hence,
 $a$ and $c$ are relatively prime, which implies that the greatest common divisor  of $a,b,c$ is $1$. Notice also that $a \ge 1$.
 
  Clearly,
 $\bar{\tau}=\left(1-\sqrt{D}/\beta\right)/2$ is the complex-conjugate of $\tau$.
 Then 
 $$\tau+\bar{\tau}=1, \quad \tau \bar{\tau}=\frac{\beta^2-D}{4\beta^2}.$$
 It follows that
 $$\tau^2+(-1) \tau+\frac{\beta^2-D}{4\beta^2}=0.$$
 Multiplying it by $\beta^2/\mathbf{d}$, we get
$$\frac{\beta^2}{\mathbf{d}} \tau^2+\left(-\frac{\beta^2}{\mathbf{d}}\right) \tau+\frac{\beta^2-D}{4\mathbf{d}}=0,$$
 i.e.,
  \begin{equation}
 \label{abcd}
 a \tau^2+b \tau +c=0.
  \end{equation}
 The properties of $a,b,c$ mentioned above combined with \eqref{abcd} allow us to describe explicitly
 $$O_{\tau}=\{u \in \bbQ(\tau), \  u \ [\tau,1] \subset [\tau,1]\}.$$
 Namely, by \cite[Ch. 2, Sect. 7.4, Lemma 1]{BS},
 $O_{\tau}$ coincides with
$$[1,a\tau]=\bbZ+\bbZ (a\tau)=\bbZ[a\tau]$$ and its discriminant  is $b^2-4ac$.
In order to finish the proof of (iii), it remains to notice that
$$a\tau=\frac{\beta^2}{\mathbf{d}}\cdot \frac{\beta+\sqrt{D}}{2\beta}=\frac{\beta}{\mathbf{d}} \cdot
\frac{\beta+\sqrt{D}}{2},$$
$$b^2-4ac=\left(-\frac{\beta^2}{\mathbf{d}}\right)^2-4 \frac{\beta^2}{\mathbf{d}}\cdot  \frac{\beta^2-D}{4\mathbf{d}}=
\left(\frac{\beta}{\mathbf{d}}\right)^2 D=\frac{\beta^2}{\mathbf{d}} \cdot \frac{D}{\mathbf{d}}.$$
 
 {\bf  (ii)}. Assume   that $\mathrm{Re}(\tau)=1/2$.  Recall that we are given the
 {\sl negative} integer $D \equiv 1 (\bmod 4)$  such that
 $$K=\bbQ(\tau)=\bbQ(\sqrt{D})=\bbQ\cdot 1 \oplus \bbQ\cdot \sqrt{D}\subset \bbC.$$
 Then there is a positive rational number $r\in \bbQ$ such that
 $$\tau=\frac{1}{2}+r\sqrt{D}.$$
 It follows that
 $$\Lambda_{\tau}=\bbZ\cdot 1+ \bbZ\cdot \tau=\bbZ+\bbZ \left(\frac{1}{2}+r\sqrt{D}\right) \subset \bbC.$$
 By definition of $O_{\tau}$ (see \eqref{LambdaO}),
 $\frac{1+\sqrt{D}}{2} \in O_{\tau}$ if and only if
 \begin{equation}
 \label{Otau}
 \frac{1+\sqrt{D}}{2}\left(\Lambda_{\tau}\right)\subset \Lambda_{\tau}.
 \end{equation}
 However,  since $\{1, \ 1/2+r \sqrt{D}\}$ is a basis of the $\bbQ$-vector space $\bbQ(\sqrt{D})$, there are certain rational numbers $\alpha,\beta,\gamma,\delta \in \bbQ$ such that
 \begin{equation}
 \label{Otau1}
 \frac{1+\sqrt{D}}{2}\cdot 1=\alpha+\beta  \left(\frac{1}{2}+r \sqrt{D}\right), \quad \frac{1+\sqrt{D}}{2} \left (\frac{1}{2}+r \sqrt{D}\right)=\gamma+\delta  \left(\frac{1}{2}+r \sqrt{D}\right).
 \end{equation}
 The inclusion \eqref{Otau} is equivalent to the condition
 \begin{equation}
 \label{ABCD}
 \alpha, \beta, \gamma, \delta \in \bbZ.
 \end{equation}
 
 Opening the brackets and collecting terms in \eqref{Otau1}, we get 
 $$\frac{1+\sqrt{D}}{2}=\left(\alpha+\beta /2\right)+\beta r \sqrt{D}; \quad \left(\frac{1}{4}+rD/2\right)+\frac{1}{2}\left(r+\frac{1}{2}\right)\sqrt{D}=\left(\gamma+\delta/2\right)+\delta r \sqrt{D}.$$
 Taking the real parts and ``imaginary'' parts, we get
 \begin{equation}
 \label{ABCDq}
 1/2=\alpha+\beta/2,  \ 1/2= \beta r; \quad 1/4+rD/2=\gamma+\delta/2; \quad \frac{1}{2}\left(r+1/2\right)=\delta r.
 \end{equation}
 The second equality of \eqref{ABCDq} implies that
 \begin{equation}
 \label{BetaR}
 r=\frac{1}{2\beta}, \quad \tau=\frac{1}{2}+\frac{1}{2\beta}\sqrt{D}.
 \end{equation}
 Since $r$ is positive, $\beta$ is also positive.
 
 It follows from  first equality of \eqref{ABCDq} that {\sl $\alpha$ is an integer if and only if $\beta$ is an odd integer}.
 This implies that if $\beta$ is {\sl not} an odd integer then \eqref{ABCD} does not hold.  So, in the course of the proof we may and will assume that
 \begin{equation}
 \label{betaOdd}
 \beta \in 1+2\bbZ, \quad \alpha \in \bbZ.
 \end{equation}
 Now let us change the strategy and  apply the already proven (iii),
 instead of investigating the integrality of $\gamma$ and $\delta$. We get 
 \begin{equation}
 \label{OtauD}
 O_{\tau}=
\bbZ\cdot 1+\bbZ\cdot \left(\frac{\beta^2}{\mathbf{d}}\right) \left(\frac{\beta+\sqrt{D}}{2\beta}\right)
\end{equation}
where $\mathbf{d}\ge 1$ is the GCD of $D$ and $\beta^2$. We need to find when
$(1+\sqrt{D})/2\in O_{\tau}$. 

Let us put
\begin{equation}
\label{wBeta}
w_{\beta}:=\frac{\beta+\sqrt{D}}{2}=\beta  \cdot \frac{\beta+\sqrt{D}}{2\beta}\in \bbQ(\sqrt{D})=\bbQ(\tau).
\end{equation}
Since $\beta$ is odd, 
 $$m_{\beta}:=(\beta-1)/2 \in \bbZ, \ \beta=2 m_{\beta}+1.$$
 This implies that
$$w_{\beta}=\frac{2 m_{\beta}+1+\sqrt{D}}{2}=m_{\beta}+\frac{1+\sqrt{D}}{2}
\in \frac{1+\sqrt{D}}{2}+\bbZ.$$
Since $\bbZ \subset O_{\tau}$, we conclude that $(1+\sqrt{D})/2\in O_{\tau}$ if and only if
$w_{\beta} \in O_{\tau}$.
So, in order to prove (ii), it suffices to check that $w_{\beta} \in O_{\tau}$ if and only if
$\beta$ divides $D$. Let us do it.

Combining \eqref{wBeta} and \eqref{OtauD}, we get
\begin{equation}
\label{Ow}
O_{\tau}=
\bbZ\cdot 1+\bbZ\cdot \left(\frac{\beta}{\mathbf{d}}\right)  w_{\beta}.
\end{equation}
Since $1$ and $w_{\beta}$ are obviously linearly independent over $\bbQ$, it follows from \eqref{Ow} that $w_{\beta} \in O_{\tau}$ if and only if there exists $n \in \bbZ$ such that
$$ n \cdot \frac{\beta}{\mathbf{d}}=1.$$
This means that the ratio $\mathbf{d}/\beta$ is an integer, i.e., $\beta$ divides $\mathbf{d}$.
Equivalently (in light of the very definition of 
$\mathbf{d}$), 
$\beta$ divides both $D$ and $\beta^2$.
Since $\beta$ always divides $\beta^2$, we get that $\beta$ divides $\mathbf{d}$ if and only if $\beta$ divides $D$.  This ends the proof of (ii).
 \end{proof}

 \begin{ex}
 Fix a negative integer $d$ that is congruent to $1$ modulo $4$  (e.g., we may take $d=-3$). 
 If $k$ and $n$ are positive odd integers then $k^2 n^2 d$ is also a negative integer that is
  congruent to $1$ modulo $4$. If we put 
  $$D=k^2 n^2 d, \quad\beta= n^2$$
   then $\beta$ is a positive odd integer dividing $D$. It follows
  that if we put 
  \begin{equation}
  \label{NKbeta}
  \tau(n,k):=\frac{1}{2}+\frac{1}{2\beta}\sqrt{D}=\frac{1}{2}+\frac{k\sqrt{d}}{2n}
  \end{equation}
  then the elliptic curve $\mathcal{E}_{\tau(n, k)}$ is {\sl odd}.

 Recall that $n$ and $k$ could be any odd positive integers.  Notice also that all  CM elliptic curves $\mathcal{E}_{\tau(n, k)}$ are isogenous to each other, because the imaginary quadratic field
 $$\bbQ(\tau(n,k))=\bbQ(\sqrt{d})$$
 does not depend on $n,k$.
  \end{ex}

\begin{cor}
\label{evenE}
Let $d$ be a square-free negative integer such that 
 $d \equiv 1 (\bmod 4)$. 
 If $n$ and $m$ are positive odd  integers  and 
 $$\tau=\frac{1}{2}+\frac{n}{m}\sqrt{d}$$ then $\mathcal{E}_{\tau}$ is even.
\end{cor}
\begin{proof}

In light of Example \ref{crucial} applied to $y=2\sqrt{|d|}$, it suffices to check the case $n=m=1$, i.e., we may (and will) assume that 
$$\tau=1/2+\sqrt{d}.$$
Clearly, 
$K:=\bbQ(\tau)=\bbQ(\sqrt{d})$. Suppose that $\mathcal{E}_{\tau}$ is {\sl odd}, i.e., the order  $O_{\tau}$ is odd. 
Let $\mathfrak{f}$ be the conductor of $ O_{\tau}$. By Lemma \ref{DefOd}, $\mathfrak{f}$  is odd, hence,
$$\frac{1-\mathfrak{f}}{2} \in \bbZ.$$
It follows from \cite[Ch. 2, Sect. 7, Th.1]{BS}
that
$$O_{\tau}=\bbZ+\mathfrak{f}\frac{1+\sqrt{d}}{2}\bbZ \ni \frac{1-\mathfrak{f}}{2}+\mathfrak{f}\frac{1+\sqrt{d}}{2}=\frac{1+\mathfrak{f}\sqrt{d}}{2}.$$
So,
$$\frac{1+\mathfrak{f}\sqrt{d}}{2} \in O_{\tau}.$$
Notice that 
$$\frac{1+\mathfrak{f}\sqrt{d}}{2}=\frac{1+\sqrt{D}}{2}$$
where $D=\mathfrak{f}^2 d$ is a negative integer that is also congruent to $1$ modulo $4$, because $\mathfrak{f}$ is odd.
By Lemma \ref{realDIV}, there is an {\sl odd} positive integer $\beta$ such that
$$\tau=1/2+\frac{1}{2\beta}\sqrt{D}=1/2+\frac{\mathfrak{f}}{2\beta}\sqrt{d}.$$
Since $\tau=1/2+\sqrt{d}$, we get
$\frac{\mathfrak{f}}{2\beta}=1$, i.e., $\mathfrak{f}=2\beta$. Since  $\mathfrak{f}$ is odd, we get the desired contradiction.
\end{proof}

 
 \section{$j$-invariants: the real case}
 \label{mainproof}
 
 We will need the following results that are either contained in  \cite[Ch. V, Sect. 2]{Sil} (see also \cite[Ch. 14, Sect. 4]{HM} and \cite[Ch. 1, Sect. 4.3]{Gross}) or follow readily from them. 
 \begin{itemize}
 \item
 Let us consider the {\sl disjoint} subsets $\mathfrak{T}_1$ and $\mathfrak{T}_2$ of $\mathfrak{H}$ defined as follows.
 $$\mathfrak{T}_1:=\{\mathbf{i}t \mid t \in \bbR, \ t \ge 1\}; \quad \mathfrak{T}_2:=\{\frac{1}{2}+\mathbf{i}t \mid t \in \bbR, \ t>1/2\}.$$
 Let us put
 $$\mathfrak{T}:=\mathfrak{T}_1 \cup  \mathfrak{T}_2 \subset \mathfrak{H}.$$
 If $s$ is a real number then there is {\sl precisely one} $\tau \in \mathfrak{T}$ such that
 $$j(\mathcal{E}_{\tau})=j(\tau)=s.$$
 In particular, the continuous function
 $$F: \mathfrak{T} \to \bbR, \  \tau \mapsto j(\tau)=j(\mathcal{E}_{\tau})$$
 is a bijective map \cite[Ch. V, Sect. 2, p. 417]{Sil} .
 \item
 The real-valued function
 $$f: [1/2, \infty] \to (-\infty,\infty),  \quad y \mapsto \tau=1/2+iy \mapsto  j(\tau)=j(\mathcal{E}_{\tau})$$
is continuous injective, and 
$$f(1/2)=1728, \  f(\sqrt{3}/2)=0, \ \lim_{y \to \infty} f(y)= - \infty$$
(see \cite[Ch. V, Sect. 2, p. 414]{Sil}).
It follows that $f$ is decreasing and the image of the map $f$ lies in the closed semi-infinite interval $(-\infty,1728]$.
So, one may view as a  decreasing function the map
$$f: [1/2, \infty) \to (-\infty,1728],  \quad  y \mapsto \tau=1/2+iy \to  j(\tau)=j(\mathcal{E}_{\tau}),$$
which is obviously bijective. In addition, the image of the open semi-infinite interval 
 $(1/2, \infty)$ under $f$ is the open semi-infinite interval $(-\infty,1728)$.
\item
The bijectiveness of the continuous map
$F:\mathfrak{T} \to \bbR$
implies that the set
$$\{j(\mathcal{E}_{\tau})\mid \tau \in \mathfrak{T}_1\}$$
coincides with the closed semi-infinite interval $[1728, \infty)$.
\end{itemize}

\begin{proof}[Proof of Theorem \ref{orbitR}]
Without loss of generality we may and will assume that
$E=\mathcal{E}_{\tau}$ with $\tau \in \mathfrak{T}$.

Suppose that $\mathcal{E}_{\tau}$ is odd.   In light of Lemma \ref{realDIV}(i), $\mathrm{Re}(\tau) \ne 0$. 
 Since $\tau \in \mathfrak{T}$, we get
$\mathrm{Re}(\tau) =1/2$, i.e., $\tau=\frac{1+\mathbf{i}y}{2}$ for some real positive $y$. In other words,
$$\tau \in \mathfrak{T}_2.$$
Let 
$E^{\prime}$ be any odd CM elliptic curve  with real $j(E^{\prime})$  such that $\mathcal{E}_{\tau}$ and $E^{\prime}$ are isogenous over $\bbC$.
Without loss of generality we may assume that
$$E^{\prime}=\mathcal{E}_{\tau^{\prime}}, \ j(E^{\prime})=j(\mathcal{E}_{\tau^{\prime}})=j(\tau^{\prime})$$
for a certain $\tau^{\prime} \in \mathfrak{T}$. Applying Lemma \ref{realDIV}(i) to $\mathcal{E}_{\tau^{\prime}}$ (instead of $\mathcal{E}_{\tau}$),
we conclude that 
 $\mathrm{Re}(\tau^{\prime}) \ne 0$ and therefore 
$$\tau^{\prime} \in \mathfrak{T}_2.$$
 It follows that $\mathcal{J}(\bbR,E)$ lies in the open interval $(-\infty,1728)$ and therefore
 the closure of
 $\mathcal{J}(\bbR,E)$ in $\bbR$  lies in the closed interval $(-\infty,1728]$.

 If $n, m$ are any  positive odd integers such that $$\frac{m}{n}y>1$$ then 
 $$\frac{m}{n}y/2>1/2,$$
  and the set of all
 such $\frac{m}{n}y/2$ is dense in $(1/2,\infty)$. If we put
$$\tau_{m,n}:=\frac{1+\mathbf{i}\frac{m}{n}y}{2} \in  \mathfrak{T}$$
then 
 the set of all  $\tau_{m,n}$ is a dense subset of $\mathfrak{T}_2$.
By Example \ref{crucial}, 
all CM elliptic curves $\mathcal{E}_{\tau_{m,n}}$ are odd and isogenous to $\mathcal{E}_{\tau}$. In addition,
$$f\left(\frac{m}{n}y/2\right) =j\left(\mathcal{E}_{\tau_{m,n}}\right)\in \mathcal{J}(\bbR,\mathcal{E}_{\tau}) \subset (-\infty,1728].$$
It follows from the continuity of $f$ that the set of all $j(\mathcal{E}_{\tau_{m,n}})$ is dense in
$f((1/2,\infty])=(-\infty,1728)$. This implies  that the closure of $\mathcal{J}(\bbR,\mathcal{E}_{\tau})$ coincides with $(-\infty,1728]$
when $\mathcal{E}_{\tau}$ is {\sl odd}.

Suppose that $\mathcal{E}_{\tau}$ is {\sl even}. Then $\bbQ(\tau)=\bbQ(\sqrt{d})$ where $d$ is a  square-free  negative integer.
  If $n, m$ are any  positive  integers such that $$\frac{m}{n}  \sqrt{|d|} > 1$$
then the CM elliptic curves $\mathcal{E}_{\frac{m}{n}\sqrt{d}}$ are  isogenous to $\mathcal{E}_{\tau}$;
in addition, they all are even, thanks to Lemma \ref{realDIV}(i).  On the other hand, obviously,
the set
of all such $\frac{m}{n}\sqrt{d}$  lies in $\mathfrak{T}_1$ (in particular, all $j\left(\mathcal{E}_{\frac{m}{n}\sqrt{d}}\right)\in \bbR$) 
and are dense there. This implies that 
the closure of $\mathcal{J}(\bbR,\mathcal{E}_{\tau})$ contains $F(\mathfrak{T}_1)=[1728, \infty)$.

Suppose that $d \not \equiv 1 (\bmod 4)$. 
Then the quadratic field
$\bbQ(\sqrt{d})$
has discriminant $4d$, which is {\sl even}. It follows from Definition \ref{DefEv}
 that every elliptic curve
$\mathcal{E}_{\tau^{\prime}}$ with $$\bbQ(\tau^{\prime})=\bbQ(\sqrt{d})=\bbQ(\tau)$$ is {\sl even} and isogenous to $\mathcal{E}_{\tau}$.
In particular, if $m,n$ are any positive integers such that 
$\frac{m}{n}  \sqrt{|d|} > 1$ then
$$\tau(m,n):=1/2+\frac{m}{n}\sqrt{d} \in \mathfrak{T}_2,  \quad \bbQ(\tau(m,n))=\bbQ(\sqrt{d}),$$ 
which implies that the elliptic curve
$\mathcal{E}_{\tau(m,n)}$ is   isogenous to $\mathcal{E}_{\tau}$ and  even. Since
$\tau(m,n) \in \mathfrak{T}_2$, we get
$$j\left(\mathcal{E}_{\tau(m,n)}\right) \in \bbR \ \text{ and therefore} \ j\left(\mathcal{E}_{\tau(m,n)}\right) \in \mathcal{J}(\bbR,\mathcal{E}_{\tau}).$$
On the other hand, the set of all such  $\tau(m,n)$ is obviously dense in $\mathfrak{T}_2$. It follows
that the closure of $\mathcal{J}(\bbR,\mathcal{E}_{\tau})$ contains  $f\left(\mathfrak{T}_1\right)=(-\infty,1728)$.
Recall that we have already checked that this closure contains $[1728, \infty)$. It follows that 
 the closure of $\mathcal{J}(\bbR,\mathcal{E}_{\tau})$ coincides with the whole $\bbR$ if $d \not\equiv 1 (\bmod 4)$.
 
 Now suppose that $d \equiv 1 (\bmod 4)$. 
 If $m,n$ are any positive odd  integers such that 
$\frac{m}{n}  \sqrt{|d|} > 1/2$ then
$$\tau(m,n):=\frac{1}{2}+\frac{m}{n}\sqrt{d} \in \mathfrak{T}_2, \quad j\left(\mathcal{E}_{\tau(m,n)}\right) \in \bbR,$$ and the elliptic curve
$\mathcal{E}_{\tau(m,n)}$ is   isogenous to $\mathcal{E}_{\tau}$. 
 By Corollary \ref{evenE}, $\mathcal{E}_{\tau(m,n)}$ is {\sl even} 
and therefore $j\left(\mathcal{E}_{\tau(m,n)}\right) \in \mathcal{J}(\bbR,\mathcal{E}_{\tau})$.
Clearly, the set of all such $\tau(m,n)$ is dense in $\mathfrak{T}_2$. It follows that
the closure of $\mathcal{J}(\bbR,\mathcal{E}_{\tau})$ contains $F(\mathfrak{T}_2)=(-\infty,1728)$.
Recall that we have already checked that this closure contains $[1728, \infty)$. It follows that 
 the closure of $\mathcal{J}(\bbR,\mathcal{E}_{\tau})$ coincides with the whole $\bbR$ if $d \equiv 1 (\bmod 4)$.
 This ends the proof of Theorem \ref{orbitR}.
\end{proof} 

\begin{rem}
\label{Rconnected}
Let $E_{\bbR}$ be an elliptic curve over $\bbR$ and let $E=E_{\bbR}\times_{\bbR}\bbC$ be its complexification.
Then there is  $\tau \in \mathfrak{H}$ that enjoys the following properties (see \cite[Ch. V,  Sect. 5.2, Th. 2.3 and its proof]{Sil}).
\begin{itemize}
\item[(i)]
$\mathrm{Re}(\tau)=0$ or $1/2$.
\item[(ii)]
$E$ is isomorphic to $\mathcal{E}_{\tau}$. In particular,
$$j(E_{\bbR})=j(E)=j(\tau)\in \bbR.$$
\item[(iii)]
$q:=\exp(2\pi \mathbf{i}\tau)$ is a nonzero real number such that $|q|<1$  and the real Lie group
$E_{\bbR}(\bbR)$ of $\bbR$-points on $E$ is isomorphic to the quotient $\bbR^{*}/q^{\bbZ}$.
(Here $q^{\bbZ}$ is the cyclic multiplicative subgroup of $\bbR^{*}$ generated by $q$.)
In particular, $E_{\bbR}(\bbR)$ is {\sl connected} if and only if $q<0$, i.e., $\mathrm{Re}(\tau)=1/2$.
\item[(iv)]
Suppose that $E$ is an {\sl odd} elliptic curve with CM. It follows from Lemma \ref{realDIV}(i) that
$\mathrm{Re}(\tau)=1/2$. In light of (iii),
$E_{\bbR}(\bbR)$ is connected.
\end{itemize}
\end{rem}

\section{Odd and even orders}
\label{OddEven}
The aim of this section is to explain why Definitions \ref{DefOd} and \ref{DefEv} make sense.
Let $O$ be an order in a quadratic field $K$ with conductor $\mathfrak{f}$.

The equivalence of properties  presented in Definition \ref{DefOd} is an immediate corollary of the following assertion.
\begin{lem}
\label{DefOdd}
The following  conditions  are equivalent.
\begin{itemize}
\item[(i)]
$\mathrm{tr}(O)=\bbZ$.
\item[(ii)]
 $\tra(O[1/n])$ contains $1$  for all   odd positive integers  $n$. 
\item[(iii)]
$\mathrm{discr}(O)$ is an odd integer.
\item[(iv)]
Both $\mathrm{discr}(K)$ and the conductor $\mathfrak{f}_O$  are odd integers.
\item[(v)]
There exists $ w\in O$ such that $\mathrm{tr}(w)$ is an odd integer.
\item[(vi)]
There exists an integer $D \equiv 1 (\bmod 4)$ that is not a square such that
$$\frac{1+\sqrt{D}}{2} \in O.$$
\item[(vii)]
There exists an integer $D \equiv 1 (\bmod 4)$ that is not a square such that
$$O=\bbZ+ \bbZ \ \frac{1+\sqrt{D}}{2} =\bbZ\left[\frac{1+\sqrt{D}}{2}\right].$$
\end{itemize}
\end{lem}

The equivalence of properties  presented in Definition \ref{DefEv} is an immediate corollary of the following assertion.
\begin{lem}
\label{DefEven}
The following  conditions  are equivalent.
\begin{itemize}
\item[(i)]
$\mathrm{tr}(O)=2\bbZ$.
\item[(ii)]
 $\tra(O[1/n])$ does {\sl not} contain $1$ for all   odd positive integers  $n$. 
\item[(iii)]
$\mathrm{discr}(O)$ is an even integer.
\item[(iv)]
Either $\mathrm{discr}(K)$ or the conductor $\mathfrak{f}_O$  is an even integer.
\item[(v)]
There exists an integer $D$ that  is not a square and such that
$$O=\bbZ+ \bbZ \ \sqrt{D}=\bbZ\left[\sqrt{D}\right].$$ 
\end{itemize}
\end{lem}

\begin{proof}[Proof of  Lemmas \ref{DefOdd} and \ref{DefEven}]

$ $\newline

{\bf Step 0} Let us start with the computation of certain discriminants. Let  $D$ be an integer that is {\sl not} a square. Then $\bbQ(\sqrt{D})$ is a quadratic field. Obviously,
\begin{equation}
\label{trRD}
\tra_{\bbQ(\sqrt{D})}\left(\sqrt{D}\right)=0.
\end{equation}
Let us consider the lattice ({\sl full module} in the terminology of \cite{BS})
$\bbZ +\bbZ \ \sqrt{D}$ in $\bbQ(\sqrt{D})$.  Let us compute its {\sl discriminant}
$\mathrm{disc}(\bbZ +\bbZ \ \sqrt{D})$ \cite[Ch. II, Sect. 2.5]{BS}.

In light of \eqref{trRD}, 
 the matrix of the trace form with respect to the basis $\{1, \sqrt{D}\}$ of this lattice is
$$\begin{pmatrix}
2 & 0\\
0 & 2D
\end{pmatrix},$$
whose determinant is $4D$. Hence, the discriminant
\begin{equation}
\label{Dev}
\mathrm{disc}\left(\bbZ +\bbZ \ \sqrt{D}\right)=4D
\end{equation}
  is {\sl even}.  On the other hand, $\bbZ+ \bbZ \ \frac{1+\sqrt{D}}{2}$ is a subgroup of index $2$ in
$$\bbZ +\bbZ \ \sqrt{D}=\bbZ+\bbZ \ (1+\sqrt{D}),$$
hence, the discriminant of  $\bbZ+ \bbZ \ \frac{1+\sqrt{D}}{2}$ is given by the formula
\begin{equation}
\label{Dodd}
\mathrm{disc}\left (\bbZ+ \bbZ \ \frac{1+\sqrt{D}}{2}\right)=
\frac{\mathrm{disc}(\bbZ +\bbZ \ \sqrt{D})}{2^2}=\frac{4D}{4}=D.
\end{equation}

Now let us discuss traces.
Recall  (\eqref{trO2} and  \eqref{disF}) that
\begin{equation}
\label{trO3}
\tra(O)=\bbZ \ \text{ or } 2\bbZ, \quad \mathrm{discr}(O)=\mathrm{discr}(K) \cdot \mathfrak{f}^2.
\end{equation}

{\bf Step 1} Let $n$ be an {\sl odd} positive integer. Then
$$\tra(O[1/n])=\bbZ[1/n] \cdot \tra(O)\subset \bbQ.$$
It follows that 
$$\tra(O[1/n])=\bbZ[1/n] \ni 1 \ \text{ if } \tra(O)=\bbZ;$$
$$\tra(O[1/n])=2\cdot \bbZ[1/n] \not\ni 1 \ \text{ if } \tra(O)=2\bbZ.$$
Now the equivalence of (i) and (ii) follows from first equality of \eqref{trO3} (for both Lemmas).
The equivalence of (iii) and (iv)  follows from the second equality of \eqref{trO3} (for both Lemmas).

{\bf Step 2} Let us concentrate for a while on the proof of   Lemma
\ref{DefOdd}.  The  equivalence of (i) and (v)  follows from first equality of \eqref{trO3}. In addition,
(vii) obviously implies (vi) while (vi) implies (v), because if we put $w=\frac{1+\sqrt{D}}{2}$ then
$$\tra(w)=\tra\left(\frac{1+\sqrt{D}}{2}\right)= \frac{1+\sqrt{D}}{2}+\frac{1-\sqrt{D}}{2}=1$$
is an {\sl odd} integer. 
In light of \eqref{Dodd}, (vii) implies that
$\mathrm{disc}(O)=D$, which is {\sl odd}.
Therefore (vii) implies (iii).

So, in order to finish the proof of  Lemma \ref{DefOdd}, it suffices to check that 
\begin{itemize}
\item[{\bf (1o) }]
(i) implies (vii);
\item[{\bf (2o) }]
 (iii) implies  (i).
\end{itemize}

Now let us switch to  Lemma \ref{DefEven}. Clearly,  (v)   implies
 (i). On the other hand, in light of \eqref{Dev}, (v) implies that $\mathrm{disc}(O)=4D$, which is {\sl even}. Therefore
(v) implies (iii).

 So, in order to finish the proof of  Lemma \ref{DefEven}, it suffices to check that 
 \begin{itemize}
\item[{\bf (1e) }]
(i) implies (v);
\item[{\bf (2e) }]
(iii)  implies (i).
\end{itemize}

{\bf Step 3} We have $K=\bbQ(\sqrt{d})$ where $d \ne 1$ is a  square-free integer. Clearly,
\begin{equation}
\label{trRoot}
\tra(\sqrt{d})=0.
\end{equation}

Suppose that $d \not\equiv 1 (\bmod 4)$. 
It is known \cite[Ch. 2, Sect. 7, Th.1]{BS} that
$$\mathrm{disc}(K)= 4d, \ \  O=\bbZ+\mathfrak{f}\sqrt{d}\ \bbZ=\bbZ+\sqrt{\mathfrak{f}^2 d}\ \bbZ, \quad \mathrm{disc}(O)=\mathrm{disc}(K)\mathfrak{f}^2= 4d \mathfrak{f}^2.$$
In particular, $\mathrm{disc}(O)$ is even. On the other hand, if we put $D:=\mathfrak{f}^2 d$ then $D$ is not a square (because $d$ is square-free) and 
$$K=\bbQ(\sqrt{d})=\bbQ(\sqrt{D}), \quad    O=\bbZ+\sqrt{D}\ \bbZ, \quad \discr(O)=4d \mathfrak{f}^2 =4D.$$
It follows  from \eqref{trRoot} that 
$\tra(O)=2\bbZ$.  This implies that $O$ enjoys properties (i), (iii), (v) of Lemma \ref{DefEven}.
On the other hand, $O$ does {\sl not} enjoy any of properties  (i)  and (iii) of Lemma \ref{DefOdd}.  In light of Step 2, this proves both Lemmas in the (last remaining) case when $d  \not\equiv 1 (\bmod 4)$.

Suppose that $d \equiv 1 (\bmod 4)$. 
It is known \cite[Ch. 2, Sect. 7, Th.1]{BS} that
$$\mathrm{disc}(K)= d, \quad O=\bbZ+\mathfrak{f}\frac{1+\sqrt{d}}{2}\bbZ, \ 
\mathrm{disc}(O)=\mathrm{disc}(K)\mathfrak{f}^2= d \mathfrak{f}^2.$$
It follows  from \eqref{trRoot} that 
\begin{equation}
\label{trRoot1}
\tra(O)=\bbZ \ \text{ if } \mathfrak{f} \ \text{ is odd};  \quad \tra(O)=2\bbZ \ \text{ if } \mathfrak{f} \ \text{ is even}.
\end{equation}
If $\mathfrak{f}$ is even  then $\tra(O)=2\bbZ$,  $\mathrm{disc}(O)$ is even, $m=\mathfrak{f}/2$ is a positive integer, and 
$$O=\bbZ+\mathfrak{f}\frac{1+\sqrt{d}}{2}\bbZ=\bbZ+m\left(1+\sqrt{d}\right)\bbZ=\bbZ+m\sqrt{d}\ \bbZ=$$
$$\bbZ+\sqrt{m^2 d}\ \bbZ=\bbZ+\sqrt{D}\ \bbZ$$
where $D:=m^2 d$, which is not a square, because $d$ is square-free. We get
$$K=\bbQ(\sqrt{d})=\bbQ(\sqrt{D}), \quad O=\bbZ+\sqrt{D}\ \bbZ, \quad  \discr(O)=\mathfrak{f}^2 d=(2m)^2 d=4D.$$
This implies that $O$ enjoys all the properties (i), (iii), (v)  of Lemma \ref{DefEven}.
On the other hand, $O$ does {\sl not} enjoy any of properties  (i)  and (iii) of Lemma \ref{DefOdd}.  
In light of Step 2, this proves both Lemmas in the case when $d  \not\equiv 1 (\bmod 4)$
and $\mathfrak{f}$ is even.

Now suppose that $\mathfrak{f}$ is {\sl odd}.  It follows from \eqref{trRoot1} that  $\tra(O)=\bbZ$ and $\mathrm{discr}(O)$ is odd. In addition, if we put $m:=(\mathfrak{f}-1)/2 \in \bbZ$ then
$\mathfrak{f}=2m+1$ and
$$O=\bbZ+\mathfrak{f}\frac{1+\sqrt{d}}{2}\bbZ=\bbZ+\frac{\mathfrak{f}+\mathfrak{f}\sqrt{d}}{2}\ \bbZ
=\bbZ+\left(m+\frac{1+\mathfrak{f}\sqrt{d}}{2}\right)\ \bbZ=$$
$$\bbZ+\frac{1+\mathfrak{f}\sqrt{d}}{2}\bbZ=\bbZ+\frac{1+\sqrt{\mathfrak{f}^2d}}{2}\ \bbZ=\bbZ+\frac{1+\sqrt{D}}{2}\bbZ$$
where $D:=\mathfrak{f}^2d$ is not a square, because $d$ is square-free. Since $\mathfrak{f}$ is odd, $\mathfrak{f}^2\equiv 1 (\bmod 4)$.
Since $d \equiv 1 (\bmod 4)$, the product $D:=\mathfrak{f}^2 d$ is also congruent to $1$ modulo $4$. We get
$$K=\bbQ(\sqrt{d})=\bbQ(\sqrt{D}), \quad O=\bbZ+\frac{1+\sqrt{D}}{2}\ \bbZ,   \quad \discr(O)=\mathfrak{f}^2 d=D.$$
This implies that $O$ enjoys properties (i), (iii), (vii) of Lemma \ref{DefOdd}.
On the other hand, $O$ does {\sl not} enjoy any of properties (i) and (iii) of Lemma \ref{DefEven}.  In light of Step 2, this proves both Lemmas in the case when $d \equiv 1 (\bmod 4)$
and $\mathfrak{f}$ is odd.  This ends the proof of both Lemmas.
\end{proof}

\begin{rem}
\label{localN}
Let $n$ be a positive odd integer, $\phi: K_1 \to K_2$ an isomorphism of quadratic fields,
and $O_1$ and $O_2$ are orders in $K_1$ and $K_2$ respectively.
Suppose that 
$$\phi(O_1[1/n])=O_2[1/n].$$
Then
$$\tra_{K_1/\bbQ}(O_1[1/n])=\tra_{K_2/\bbQ}(O_2[1/n]).$$
It follows from Lemmas \ref{DefOdd}(ii) and \ref{DefEven}(ii) that $O_1$ and $O_2$ have the same parity,
i.e., either they both are odd or both even.
\end{rem}

\begin{rem}
\label{disc4}
Recall that if $\mathfrak{f}$ is an odd integer then  $\mathfrak{f}^2\equiv 1 \ (\bmod 4)$.
On the other hand, if $\mathfrak{f}$ is an even integer then  $\mathfrak{f}^2 \equiv 0 \ (\bmod 4).$

Let $O$ be an {\sl order} in a {\sl quadratic field} $K$. Now the explicit formulas for its discriminant (see Step 3 above) imply the following.

\begin{itemize}
\item[(e)]
$O$ is {\sl even} if and only if $\mathrm{disc}(O)\equiv 0 \ (\bmod 4)$. If this is the case then there is an integer $D$ that is not a square and such that
$$K=\bbQ(\sqrt{D}), \quad O=\bbZ+ \bbZ \ \sqrt{D} , \quad \discr(O)=4D.$$
\item[(o)]
$O$ is {\sl odd} if and only if $\mathrm{disc}(O)\equiv 1 \ (\bmod 4)$. If this is the case then there is an integer $D\equiv 1 \ (\bmod 4)$ that is not a square and such that
$$K=\bbQ(\sqrt{D}), \quad O=\bbZ+ \bbZ \ \frac{1+\sqrt{D}}{2}, \quad \discr(O)=D.$$
\end{itemize}
\end{rem}

The following assertion is a natural complement to Key Lemma \ref{realDIV}(ii).

\begin{lem}
\label{key}
Let $D$ be a negative integer such that 
  $D \equiv 1 (\bmod 4)$.  Let $\beta$ be a positive  integer   that
 divides $D$. Let us put
 $$\tau=1/2+\frac{1}{2\beta}\sqrt{D}\in \mathfrak{H}.$$
 Then
 $$O_{\tau}=\bbZ\left[\frac{1+\sqrt{D}}{2}\right]=\bbZ+ \bbZ \ \frac{1+\sqrt{D}}{2}$$
 if and only if $\beta$ and $D/\beta$ are relatively prime.
 \end{lem}
 
 \begin{proof}
 Since $D$ is odd, its divisor $\beta$ is also odd.
 In light of Key Lemma \ref{realDIV}(ii),
 $$\mathcal{O}_D:=\bbZ+ \bbZ \ \frac{1+\sqrt{D}}{2}\subset O_{\tau}.$$
 We need to prove that the equality holds if and only if $\beta$ and $D/\beta$ are relatively prime. By Key Lemma \ref{realDIV}(iii),
 $$\discr(O_{\tau})=(\beta/\mathbf{d})^2 D$$
 where $\mathbf{d}$ is the GCD of $\beta^2$ and $D$. 
 
 Notice that  $\mathcal{O}_D$ is a subgroup of finite index, say, $f\ge 1$ in $O_{\tau}$. 
 Clearly, $\mathcal{O}_D=O_{\tau}$ if and only if $f=1$.
 However,
 $$D=\discr(\mathcal{O}_D)=f^2 \discr(O_{\tau})=f^2 (\beta/\mathbf{d})^2 D,$$
 i.e.,
 $$f^2 (\beta/\mathbf{d})^2=1,$$
 which means that 
 $f=\mathbf{d}/\beta$ (recall that all $f,\beta, \mathbf{d}$ are positive integers). It follows that $f=1$ if and only if $\mathbf{d}=\beta$.
 
Since $\mathbf{d}$ is the GCD of $\beta^2$ and $D$, the integers $\beta^2/\mathbf{d}$ and $D/\mathbf{d}$ are relatively prime. So, if $\mathbf{d}=\beta$ then
$\beta=\beta^2/\beta$ and $D/\beta$ are relatively prime.

Conversely, suppose that the integers $\beta$ and $\tilde{\beta}=D/\beta$ are relatively prime. This implies that
 the GCD of $\beta^2 \ (=\beta \cdot \beta)$ and $D \ (=\beta \cdot \tilde{\beta})$ is $\beta$, i.e., $\mathbf{d}=\beta$. This ends the proof.
 \end{proof}

\section{Isogenies of Elliptic Curves}
\label{OddEvenE}
\begin{proof}[Proof of Proposition \ref{isogPar}]
There is a (dual) isogeny $\psi: E_2\to E_1$ such that $\psi\circ \phi: E_1 \to E_1$ is multiplication by $n$ in $E_1$
and $\phi\circ \psi: E_2 \to E_2$ is multiplication by $n$ in $E_2$ \cite[p. 7]{Gross}. Then we get an isomorphism of $\bbQ$-algebras
$$\Phi: \End^0(E_1)=\End(E_1)\otimes \bbQ \to \End(E_2)\otimes \bbQ=\End^0(E_2),   \ u_1 \mapsto \frac{1}{n} \phi u_1\psi,$$
whose inverse is an isomorphism of $\bbQ$-algebras
$$\Psi: \End^0(E_2)=\End(E_2)\otimes \bbQ \to \End(E_1)\otimes \bbQ= \End^0(E_1),   \ u_2 \mapsto \frac{1}{n} \psi u_2\phi.$$
Clearly, both $\Phi$ and $\Psi$ sends $1$ to $1$, sums to sums and therefore are homomorphisms of $\bbQ$-vector spaces. On the other hand
if $u_1 \in \End^0(E_1)$ then
$$\Psi(\Phi(u_1))= \frac{1}{n} \psi \Phi(u_1)\phi=\frac{1}{n} \psi \left( \frac{1}{n} \phi u_1\psi\right)\phi=\frac{1}{n} \left(\frac{1}{n} \psi\phi\right)u_1
\left( \frac{1}{n}\psi \phi\right)=u_1.$$
Since both $\End^0(E_1)$ and $\End^0(E_2)$ are $\bbQ$-vector spaces of the same dimension 2, $\Phi$ is an isomorphism of $\bbQ$-vector spaces and $\Psi$ is its
inverse, hence, also an isomorphism of $\bbQ$-vector spaces.
The only thing that remains to check  is that both $\Phi$ and $\Psi$ are compatible with multiplication. Let us check it.
If $u_1,v_1 \in \End^0(E_1)$ then
$$\Phi(u_1)\Phi(v_2)=\frac{1}{n} \phi u_1 \psi \frac{1}{n} \phi  v_1\psi=\frac{1}{n} \phi u_1 \frac{1}{n} (\psi  \phi)  v_1\psi=\frac{1}{n} \phi u_1 v_1\psi=\Phi(u_1v_1).$$
This proves  that $\Phi$ is a homomorphism of $\bbQ$-algebras and therefore is an isomorphism of quadratic fields $\End^0(E_1)$ and $\End^0(E_2)$.
Hence, its inverse $\Psi$ is also a field isomorphism.

Now I claim that
\begin{equation}
\label{End1n}
\Phi(\End(E_1)[1/n])= \End(E_2)[1/n], \quad \Psi(\End(E_2)[1/n])= \End(E_1)[1/n].
\end{equation}
Indeed, it follows from the definition of $\Phi$ and $\Psi$ that
$$\Phi(\End(E_1))\subset \frac{1}{n}  \End(E_2)\subset  \End(E_2)[1/n],$$
$$\Psi(\End(E_2))\subset 
 \frac{1}{n}  \End(E_1)\subset  \End(E_1)[1/n].$$
 It follows that
 $$\Phi(\End(E_1)[1/n])\subset \End(E_2)[1/n], \quad \Psi(\End(E_2)[1/n])\subset \End(E_1)[1/n]$$
and therefore
$$\End(E_1)[1/n])=\Psi(\Phi(\End(E_1)[1/n]))\subset \Psi(\End(E_2)[1/n]),$$
$$\End(E_2)[1/n]=  \Phi(\Psi(\End(E_2)[1/n]\subset \Phi(\End(E_1)[1/n]).$$
We get
$$\End(E_1)[1/n])\subset \Psi(\End(E_2)[1/n])]\subset \End(E_1)[1/n],$$
$$\End(E_2)[1/n]\subset \Phi(\End(E_1)[1/n])\subset \End(E_2)[1/n],$$
which proves \eqref{End1n}.

So, $\Phi$ is an isomorphism of quadratic fields  $\End^0(E_1)$ and $\End^0(E_2)$ such that
$$\Phi(\End(E_1)[1/n])= \End(E_2)[1/n].$$
Now the desired result follows from Remark \ref{localN}.
\end{proof}

\begin{proof}[Proof of Proposition \ref{isogR}]
We may assume that
$$E_1 = E_{1,\bbR} \times_{\bbR}\bbC, \quad E_2=E_{2,\bbR} \times_{\bbR}\bbC.$$
Let $\phi: E_1 \to E_2$ be an isogeny and $\bar{\phi}: E_1 \to E_2$ its ``complex-conjugate''.
If $\bar{\phi}=\phi$ then $\phi$ could be descended to an isogeny $E_{1,\bbR}\to E_{2,\bbR}$ over $\bbR$.

Suppose that $\bar{\phi}\ne \phi$. Then its difference 
$\psi:=\bar{\phi}- \phi: E_1 \to E_2$ is a nonzero homomorphism of elliptic curves and therefore is an isogeny.
Its complex-conjugate
$\bar{\psi}=\overline{\bar{\phi}- \phi}$ coincides with $\phi-\bar{\phi}=-\psi$.  On the other hand,  it is well known that {\sl not} all endomorphisms
of the CM elliptic curve $E_1$ are defined over $\bbR$, i.e., there is an endomorphism $\alpha$ of $E_1$ such that its complex-conjugate
$\bar{\alpha}$ does {\sl not} coincide with $\alpha$ (see Lemma \ref{CRCM} below). 
 This means that $\beta=\bar{\alpha}-\alpha: E_1 \to E_1$ is a nonzero endomorphism of $E_1$,
i.e., is an isogeny. Clearly, the complex conjugate $\bar{\beta}$ of $\beta$ coincides with
$$\overline{\bar{\alpha}- \alpha}=\alpha-\bar{\alpha}=-\beta.$$
So,
$$\bar{\psi}=-\psi, \quad \bar{\beta}=-\beta.$$
This implies that the composition $\lambda:=\psi \circ\beta: E_1 \to E_2$ is an isogeny, whose complex conjugate $\bar{\lambda}=\overline{\psi \circ\beta}$ equals
$$\bar{\psi}\circ \bar{\beta}=(-\psi)\circ (-\beta)=\psi\circ \beta=\lambda.$$
So,  the nonzero homomorphism $\lambda$ coincides with its ``complex-conjugate'',
 i.e., is defined over $\bbR$. This means that
there is a $\bbR$-homomorphism $\lambda_{\bbR}: E_{1,\bbR} \to E_{2,\bbR}$ of elliptic curves, whose ``complexification'' $E_1 \to E_2$ coincides with 
$\lambda$.  Since $\lambda \ne 0$, we obtain that  $\lambda_{\bbR} \ne 0$ and therefore is an {\sl isogeny}.
\end{proof}
The following assertion (and its proof that I am going to reproduce) is pretty well known but I was unable to find a reference.
\begin{lem}
\label{CRCM} 
Let $E_{\bbR}$ be an elliptic curve over $\bbR$.  Then the ring $\End_{\bbR}(E_{\bbR})$  of $\bbR$-endomorphisms of $E_{\bbR}$ is $\bbZ$
and the corresponding endomorphism algebra
$$\End_{\bbR}^{0}(E_{\bbR}):=\End_{\bbR}(E_{\bbR})\otimes \bbQ$$
is $\bbQ$.
\end{lem}

\begin{proof}
Let us consider the complexification $E=E_{\bbR}\times_{\bbR}\bbC$ of $E_{\bbR}$. We have
$$\bbZ\subset \End_{\bbR}(E_{\bbR})\subset \End(E), \quad \bbQ\subset  \End_{\bbR}^{0}(E_{\bbR})\subset \End^{0}(E).$$
If $\End_{\bbR}(E_{\bbR})\ne \bbZ$ then 
$$\End(E)\ne \bbZ, \quad   \End^{0}(E) \ne \bbQ,$$ i.e., $E$ is an elliptic curve with CM
and $K:= \End^{0}(E)$ is an imaginary {\sl quadratic} field. This implies that
$$K= \End^{0}(E)=\End_{\bbR}^{0}(E_{\bbR})$$
and $\End_{\bbR}(E_{\bbR})$ is an order in the quadratic field $K$. Let $\Omega^1_{\bbR}(E_{\bbR})$ be the space of  differentials of the first kind on $E_{\bbR}$,
which is a  $\bbR$-vector space of dimension $1$. By functoriality, $\End_{\bbR}(E_{\bbR})$  
 acts on $\Omega^1_{\bbR}(E_{\bbR})$,
which gives us the injective  ring homomorphism
$$\delta:\End_{\bbR}(E_{\bbR}) \hookrightarrow \End_{\bbR}(\Omega^1_{\bbR}(E_{\bbR}))=\bbR$$
\cite[Ch. I, Sect. 2.8]{ShimuraT}.  By $\bbQ$-linearity, $\delta$ extends  to the homomorphism of $\bbQ$-algebras
$$\delta_{\bbQ}:K=\End_{\bbR}^{0}(E_{\bbR})=\End_{\bbR}(E_{\bbR})\otimes\bbQ\to \End_{\bbR}(\Omega^1_{\bbR}(E_{\bbR}))=\bbR,$$
which is also injective (since $K$ is a field). So, $\delta_{\bbQ}: K \to \bbR$ is a field embedding. But such an embedding does {\sl not} exist, 
since $K$ is an {\sl imaginary} quadratic field. The obtained contradiction proves the desired result.
\end{proof}

\section{Odd  elliptic curves revisited}
\label{revisit}
The aim of this section to ``classify'' (count) all  CM elliptic curves  with given odd discriminant $D$ and real $j$-invariant.
 We will do it in Theorem \ref{endoD},  using
Lemma \ref{key} and
``classifying'' (counting  the number of)  positive divisors $\beta$ of  $D$
that are relatively prime to $D/\beta$. Let us start with the following definition (notation).

\begin{defn}[Definition-Notation]
Let $n$ be a nonzero integer. We write $P_n$ for the (finite) set of prime divisors of $n$. By the Main Theorem of Arithmetic,
$$n= \pm \prod_{p \in P_n} p^{e_{n,p}}$$
where $e_{n,p}$ are certain positive integers uniquely determined by $n$.

 Let us call a divisor $r$ of $n$ a {\sl saturated} divisor if the integers $r$ and $n/r$ are relatively prime.
 
 Clearly, $r$ is a saturated divisor of $n$ if and only if $|n|/r$ is one. It is also clear  that $r$ is a saturated divisor of $n$
 if and only if $-r$ is also one.
\end{defn}

The following assertion is an easy exercise in elementary number theory that will be proven at the end of this section.
\begin{prop}
\label{fullD}
The following statements hold.
\begin{itemize}
\item[(i)]
Let $S$ be a subset of $P_n$.  Then both  integers $$\mathbf{r}_{n,S}=  \prod_{p \in S} p^{e_{n,p}}$$ and $-\mathbf{r}_{n,S}$ 
are saturated divisors of $n$.
(As usual, if $S$ is the empty set then $r_{n, \emptyset}=1$.)
Conversely, if $c$ is a saturated divisor of $n$ then  $-c$ is also a saturated divisor of $n$ and
$$c= \pm  \prod_{p \in P_c} p^{e_{n,p}}=\pm \mathbf{r}_{n,P_c}.$$
\item[(ii)]
The number of positive saturated divisors of $n$ is the number of subsets of $P_n$, i.e., $2^{\#(P_n)}$
where $\#(P_n)$ is the cardinality of $P_n$, i.e., the number of prime divisors of $n$.
\item[(iii)]
If $S_1$ and $S_2$ are disjoint subsets of $P_n$ then $\mathbf{r}_{n, S_1}$ and $\mathbf{r}_{n, S_2}$ are relatively prime and
$$\mathbf{r}_{n, S_1\cup S_2}=\mathbf{r}_{n, S_1} \mathbf{r}_{n, S_2}.$$
\item[(iv)]
Suppose that $n$ is a negative integer that is congruent to $1$ modulo $4$. Then the number
of positive saturated divisors $r$ of $n$ with $r<\sqrt{|n|}$ equals $2^{\#(P_n)-1}$.
\end{itemize}
\end{prop}

\begin{thm}
\label{endoD}
Let $D$ be a negative integer that is congruent to $1$ modulo $4$. 

\begin{enumerate}
\item[(1)]
Let $E$ be an elliptic curve  over $\bbC$ with $j(E)\in \bbR$.  Then the following two conditions are equivalent.
\begin{itemize}
\item[(i)]
$E$ has CM and its endomorphism ring
$\End(E)$ is an order of discriminant $D$. (In particular, $E$ is odd.)
\item[(ii)]
There exists a positive 
saturated 
divisor $\beta$ of $D$ such that 
$\beta<\sqrt{|D|}$ and
$$j(E)=j(\tau) \ \text{where} \ \tau=\frac{1}{2}+\frac{\sqrt{D}}{2\beta}=\frac{1}{2}+\frac{\sqrt{|D|}}{2\beta}\mathbf{i}\in \mathfrak{T}_2\subset  \mathfrak{H}.$$
\end{itemize}
\item[(2)]
Let $\mathcal{J}_D(\bbR)$ be the set of $j$-invariants of all elliptic curves $E$ that enjoy the equivalent properties (i)-(ii) of (1).
Let $s_D$ be the number of positive prime divisors of $D$.
Then $\mathcal{J}_D(\bbR)$ consists of $2^{s_D-1}$ distinct real numbers and lies in the semi-open interval
$\left[j\left(\frac{1+\sqrt{D}}{2}\right), 1728\right)$.
\end{enumerate}
\end{thm}

\begin{proof}

(1) Suppose that $E$ enjoys the properties (i), i.e.,  $j(E) \in \bbR$, $E$ has CM and the order $\End(E)$ has discriminant $D$.
Since $j(E) \in \bbR$, there is (precisely) one $\tau \in \mathfrak{T}$ such that $E$ is isomorphic to $\mathcal{E}_{\tau}$
 and $j(E)=j(\tau)$.  So, we may assume that $E=\mathcal{E}_{\tau}$.
 In particular,  $\mathcal{E}_{\tau}$ is odd. It follows from Key Lemma \ref{realDIV}(i) that $\mathrm{Re}(\tau) \ne 0$,
 i.e., 
 $$\tau \in \mathfrak{T}_2, \ \mathrm{Re}(\tau)=\frac{1}{2}.$$
 Since $O=\End(\mathcal{E}_{\tau}) \cong O_{\tau}\subset \bbC$ is an order with {\sl odd} discriminant $D$, it follows from Remark \ref{disc4} that
 \begin{equation}
 \label{revisitO}
 O_{\tau}=\bbZ\left[\frac{1+\sqrt{D}}{2}\right]=\bbZ+\bbZ \ \frac{1+\sqrt{D}}{2}=\mathcal{O}_D\subset \bbC.
 \end{equation}
 Combining  Lemmas \ref{realDIV} and \ref{key}, we conclude that \eqref{revisitO} holds if and only if
  there is a  saturated positive divisor $\beta$ of $D$ (each divisor of $D$ is obviously odd) such that
 $$ \tau=\frac{1}{2}+\frac{\sqrt{D}}{2\beta}=\frac{1}{2}+\frac{\sqrt{|D|}}{2\beta}\mathbf{i}.$$
 Since $\tau \in \mathfrak{T}_2$, the imaginary part of $\tau$ is strictly greater than $1/2$, i.e,
 $$\beta<\sqrt{|D|}.$$
 This proves the equivalence of (i) and (ii).
 
 (2) In order to prove the second assertion of our theorem,
              recall that the function
 $$[1/2, \infty] \to \bbR, \quad t \mapsto j(1/2+\mathbf{i}t )$$
 is {\sl strictly} decreasing and $j(1/2+ \mathbf{i}/2)=1728$. 
 Now the desired result  follows readily from Proposition \ref{fullD}(iv).
\end{proof}

\begin{proof}[Proof of Proposition \ref{fullD}]

$ $\newline

(i) Let $S^{\prime}=P_n \setminus S$. Then 
$$n=\pm \mathbf{r}_{n,S} \mathbf{r}_{n, S^{\prime}}, \quad \frac{n}{\mathbf{r}_{n,S}}=\pm \mathbf{r}_{n,S^{\prime}}.$$
 If $p$ is a prime divisor of  both $\mathbf{r}_{n,S}$ and $\mathbf{r}_{n,S^{\prime}}$
then it must belong to both sets $S$ and $S^{\prime}$. However,   $S$ and $S^{\prime}$ do {\sl not} meet each other. Hence,
$\mathbf{r}_{n,S}$ and $\mathbf{r}_{n,S^{\prime}}=\pm n/\mathbf{r}_{n,S}$ have no common prime factors and therefore are relatively prime. This implies that the divisor
$\mathbf{r}_{n,S}$ (and therefore $-\mathbf{r}_{n,S}$)  of $n$ is saturated.

Conversely, suppose that $c$ is a saturated divisor of $n$. Then $P_c \subset P_n$ and
$$c=\pm \prod_{p\in P_r}p^{e_{c,p}}, \quad 1 \le e_{c,p}\le e_{n,p} \ \forall p \in P_c.$$
I claim that  
$$e_{c,p}=e_{n,p} \ \forall p \in P_c,  \ \text{i.e.}, \  d=\pm \mathbf{r}_{P_c}.$$
Indeed, suppose that $e_{c,p}<e_{n,p}$ for some $p\in P_d$. Then $p$ divides both $c$ and $n/c$,
which contradicts the saturatedness of $c$. This ends the proof.

(ii) It follows from (i) that the set of positive saturated divisors of $n$ coincides with 
$$\{\mathbf{r}_{n,S} \mid S \subset P_n\}.$$ 
It is also clear that if $S_1$ and $S_2$ are distinct subsets of $P_n$ then $\mathbf{r}_{n,S_1} \ne \mathbf{r}_{n,S_2}$,
because  their sets of prime divisors (i.e., $S_1$ and $S_2$) do {\sl not} coincide.
So, the number of positive saturated divisors of $n$ equals the number of subsets of $P_n$, which is $2^{\#(P_n)}$.

(iii) Since $S_1$ (resp. $S_2$) is the set of prime divisors of $\mathbf{r}_{n,S_1}$
(resp.$\mathbf{r}_{n,S_2}$) and $S_1$ does {\sl not} meet $S_2$,  the integers
 $\mathbf{r}_{n,S_1}$ and $\mathbf{r}_{n,S_2}$ have no common prime factors and therefore are relatively prime.  On the other hand,
 $$\mathbf{r}_{n,S_1\cup S_2}=\prod_{p\in S_1\cup S_2} p^{e_{n,p}}=
 \left(\prod_{p\in S_1} p^{e_{n,p}}\right)  \left(\prod_{p\in S_2} p^{e_{n,p}}\right)=
 \mathbf{r}_{n,S_1} \mathbf{r}_{n,S_2}.$$

(iv) Since $n \equiv 1 (\bmod 4)$, the positive odd integer $|n|=(-n)$ is congruent to $3$ modulo $4$
and therefore is {\sl not} a square. Hence, if $\beta$ is a positive saturated divisor of $n$ then
$|n|/\beta$ is also a positive saturated divisor of $n$ while $\beta \ne |n|/\beta$. Clearly,
 there is precisely one element in the pair $\{\beta, |n|/\beta\}$ that is strictly less than $\sqrt{|n|}$.
In light of (ii), the number of positive saturated divisors of $n$ that satisfy this inequality is
 $2^{\#(P_n)}/2= 2^{\#(P_n)-1}$.
\end{proof}

\end{document}